\documentclass[11pt]{amsart}

\usepackage{amsfonts,amsmath,amssymb,amsthm}
\allowdisplaybreaks[4]

\usepackage[colorlinks,linkcolor=red,citecolor=green]{hyperref}
\usepackage{makecell}
\usepackage{braket}
\usepackage{graphics}
\usepackage{xcolor}
\usepackage{tikz}
\newcommand{\circled}[1]{
  \tikz[baseline=(char.base)]{
    \node[shape=circle,draw,inner sep=1pt] (char) {#1};
  }
}

\newcommand{\dd}{\mathrm{d}}

\newtheorem{theorem}{Theorem}
\newtheorem{lemma}[theorem]{Lemma}

\newtheorem{proposition}[theorem]{Proposition}

\numberwithin{theorem}{section}
\numberwithin{equation}{section}

\begin{document}

\setlength{\baselineskip}{1.1\baselineskip}

\title[Power convexity of the torsion function in hyperbolic space]{Power convexity of the torsion function on horo-convex domains in hyperbolic space}

\author{Wei Zhang}
\address{School of Mathematics and Statistics\\
Lanzhou University\\
Lanzhou, 730000, Gansu Province, China.}
\email{zhangw@lzu.edu.cn}

\author{Qi Zhou}
\address{School of Mathematics and Statistics\\
Lanzhou University\\
Lanzhou, 730000, Gansu Province, China.}
\email{zhouqi2025@lzu.edu.cn, zhouqimath20@lzu.edu.cn}

\date{\today}

\maketitle

\begin{abstract}
In this paper, we study the torsion problem on bounded, smooth, strictly horo-convex domains in hyperbolic space. We prove that if the diameter of the domain is sufficiently small, then $-\sqrt{u}$ is strictly convex, where $u$ denotes the torsion function. The main ingredients of our proof are the constant rank theorem, a boundary convexity estimate, and a deformation argument for the domain. We also show that the exponent $1/2$ is optimal, even among strictly horo-convex domains of arbitrarily small diameter.
\end{abstract}



\section{Introduction}

Let $\Omega$ be a bounded domain in a Riemannian manifold $(M^n,g)$. The torsion problem is the Dirichlet boundary value problem
\begin{equation}\label{Equation-torsion problem}
\begin{cases}
\Delta u=-2 & \mathrm{in}\ \Omega,\\
\hspace{0.3cm}u=0 & \mathrm{on}\ \partial\Omega,
\end{cases}
\end{equation}
whose solution $u$ is called the torsion function of $\Omega$. A fundamental question is whether the convexity of the domain is inherited by $u$. In the Euclidean setting, the classical theorem of Makar-Limanov \cite{Makar-Limanov1971} asserts that $-\sqrt{u}$ is convex on a bounded convex planar domain.

The analogous model problem in spectral geometry is the first Dirichlet eigenvalue problem
\begin{equation}\label{Equation-eigenvalue problem}
\begin{cases}
\Delta\psi=-\lambda_1\psi & \mathrm{in}\ \Omega,\\
\hspace{0.3cm}\psi=0 & \mathrm{on}\ \partial\Omega,
\end{cases}
\end{equation}
where $\psi>0$ in $\Omega$, for which Brascamp and Lieb \cite{Brascamp-Lieb1976} proved that $\log\psi$ is concave on convex domains in $\mathbb{R}^n$. A quantitative refinement of this property played an essential role in the proof of the fundamental gap conjecture by Andrews and Clutterbuck \cite{Andrews-Clutterbuck2011}.

Several methods were subsequently developed to study these two problems. Korevaar \cite{Korevaar1983b} introduced the concavity maximum principle for quasilinear elliptic and parabolic equations by applying the classical maximum principle to a suitable concavity function. Kennington \cite{Kennington1985} further refined this method and used the resulting principle to establish power convexity results for a broad class of boundary value problems. In particular, Kennington pointed out that the concavity number $1/2$ of $u$ is sharp in \eqref{Equation-torsion problem}; see also the monograph of Kawohl \cite{Kawohl1985}. A different approach, based on convex envelopes and the comparison principle, was introduced by Alvarez, Lasry, and Lions \cite{Alvarez-LL1997}.

Another influential tool for establishing the convexity of solutions is the so-called microscopic convexity principle, also known as the constant rank theorem. It was introduced and developed in the 1980s by Caffarelli and Friedman \cite{Caffarelli-Friedman1985}, Singer, Wong, Yau, and Yau \cite{Singer-WYY1985}, and Korevaar and Lewis \cite{Korevaar-Lewis1987}. The central idea is to prove that the Hessian of a convex solution has constant rank. For the torsion and first eigenvalue problems, this method is applied to $-\sqrt{u}$ and $-\log\psi$, respectively. When strict convexity is already known near the boundary and for one member of a deformation family, the constant rank theorem and the continuity method yield strict convexity throughout the domain. This microscopic convexity principle was further developed for fully nonlinear elliptic equations through the work of Caffarelli, Guan, and Ma \cite{Caffarelli-GM2007}, Bian and Guan \cite{Bian-Guan2009}, Sz\'ekelyhidi and Weinkove \cite{Szekelyhidi-Weinkove2016}, and others.

Convexity properties on space forms have also received considerable attention. In particular, the $\log$-concavity of the first eigenfunction, namely the solution to \eqref{Equation-eigenvalue problem}, has been extensively studied. On the sphere $\mathbb{S}^n$, Lee and Wang \cite{Lee-Wang1987} proved the $\log$-concavity of the first Dirichlet eigenfunction on convex domains by adapting the continuity argument of  Singer, Wong, Yau, and Yau. The sharp fundamental gap estimate for convex domains in the sphere was later established by He, Wei, and Zhang \cite{He-WZ2020} for $n\geq 3$ and by Dai, Seto, and Wei \cite{Dai-SW2021} for $n=2$. For convex domains in surfaces of positive curvature satisfying a condition involving the curvature and its derivatives, Khan, Nguyen, Tuerkoen, and Wei \cite{Khan-NTW2025} obtained strong $\log$-concavity estimates for the first eigenfunction and corresponding lower bounds for the fundamental gap.

Concerning the torsion problem, Korevaar \cite{Korevaar1987} reported unpublished joint work with Treibergs extending the constant rank theorem to spaces of nonnegative constant curvature and indicated its applications to the torsion and first eigenvalue problems on the sphere. More recently, Grossi, Provenzano, and Raom \cite{Grossi-PRarXiv2026} gave a detailed treatment of the Lane-Emden equation on uniformly convex domains in $\mathbb{S}^2$ and established the corresponding strict power convexity properties, with the torsion problem \eqref{Equation-torsion problem} as a particular case.

The situation changes substantially in negative curvature. Shih \cite{Shih1989} constructed a convex domain in the hyperbolic plane $\mathbb{H}^2$ whose first Dirichlet eigenfunction is not $\log$-concave. Bourni, Clutterbuck, Nguyen, Stancu, Wei, and Wheeler \cite{Bourni-CNSWW2022} subsequently showed that such counterexamples exist in the $n$-dimensional hyperbolic space $\mathbb{H}^n$ for every $n\geq 2$. These results show that geodesic convexity is insufficient to recover the corresponding Euclidean conclusions. A useful stronger notion is horo-convexity. A domain $\Omega\subset\mathbb{H}^n$ is called horo-convex if at every point $p\in\partial\Omega$, there exists a horosphere through $p$ such that $\Omega$ lies on its convex side. For a smooth domain, this is equivalent to $$\mathrm{II}_{\partial\Omega}\geq g_{\partial\Omega},$$ where $\mathrm{II}_{\partial\Omega}$ denotes the second fundamental form with respect to the outward unit normal and $g_{\partial\Omega}$ denotes the induced metric. The domain is strictly horo-convex if the above inequality is strict. In this setting, Wei and Xiao \cite{Wei-Xiao2025} used a constant rank theorem to establish the $\log$-concavity of the first eigenfunction on horo-convex domains of sufficiently small diameter in $\mathbb{H}^n$. In dimension two, the diameter restriction was later removed by Dai, Ennis, Nguyen, and Wei \cite{Dai-ENW2026}, who further proved that every superlevel set of the first eigenfunction is horo-convex.

Motivated by these developments, we study the torsion problem \eqref{Equation-torsion problem} on bounded domains in $\mathbb{H}^n$. The following is the main result of this paper.

\begin{theorem}\label{Theorem-main theorem}
Let $\Omega$ be a bounded, smooth, strictly horo-convex domain in $\mathbb{H}^n$ for $n\geq 2$ with diameter $\mathrm{diam}(\Omega)\leq r_0(n)$, where $r_0(n)>0$ is a constant to be determined below. If $u$ is the solution of the torsion problem \eqref{Equation-torsion problem}, then $-\sqrt{u}$ is strictly convex in $\Omega$.
\end{theorem}

The proof below allows us to take $r_0(n)=1/(2\sqrt{n})$. The exponent $1/2$ in Theorem \ref{Theorem-main theorem} is also optimal. More precisely, we show in Section \ref{Section 7} that for any $n\geq 2$, $\alpha>1/2$, and $\varepsilon>0$, there exists a bounded, smooth, strictly horo-convex domain $\Omega\subset\mathbb{H}^n$ with $\mathrm{diam}(\Omega)<\varepsilon$ such that $u^\alpha$ is not concave, where $u$ denotes its torsion function.

We next outline the proof of Theorem \ref{Theorem-main theorem}. Set $v=-\sqrt{u}$ and, following Wei and Xiao \cite{Wei-Xiao2025}, define
\begin{equation}\label{Tensor}
\Lambda=\nabla^2v-|\nabla v|g,
\end{equation}
where $\nabla v$ and $\nabla^2v$ denote the gradient and Hessian of $v$ with respect to the hyperbolic metric $g$, respectively. The tensor $\Lambda$ is naturally adapted to horo-convexity, since its estimate near the boundary $\partial\Omega$ involves precisely the difference $\mathrm{II}_{\partial\Omega}-g_{\partial\Omega}$. Our aim is to prove that $\Lambda>0$ throughout $\Omega$, which immediately implies that $v$ is strictly convex. 

A direct computation shows that $v$ satisfies
\begin{equation}\label{Equation-torsion-v}
\Delta v=-\frac{1+|\nabla v|^2}{v}.
\end{equation}
The key to proving the positivity of $\Lambda$ is to establish a constant rank theorem for $\Lambda$. Its proof is more involved than that for the eigenvalue problem \eqref{Equation-eigenvalue problem} considered by Wei and Xiao \cite{Wei-Xiao2025}, owing to the explicit dependence on $v$ in equation \eqref{Equation-torsion-v}. The eigenvalue estimate used in their work is replaced here by the uniform estimate
\begin{equation}\label{Estimate-uniform estimate-v}
-\frac{1}{2n}<v<0\quad \mathrm{in}\ \Omega.
\end{equation}
To derive \eqref{Estimate-uniform estimate-v}, we first study the torsion function on a sufficiently small geodesic ball and then apply the comparison principle. On such balls, the positivity of $\Lambda$ can also be verified directly. For a general domain, we prove that $\Lambda$ is positive definite near $\partial\Omega$ and establish the constant rank theorem under the assumption \eqref{Estimate-uniform estimate-v}. Finally, we construct a family of horo-convex domains joining the given domain to a geodesic ball. Along this deformation, any interior degeneration of $\Lambda$ would force its rank to be constant throughout the domain, contradicting its positivity near the boundary. The continuity method therefore yields $\Lambda>0$ in $\Omega$.

The paper is organized as follows. In Section \ref{Section 2}, we provide some preliminaries, including elementary symmetric functions and the covariant derivative formulas in hyperbolic space. In Section \ref{Section 3}, we study the torsion function on geodesic balls and derive the uniform estimate, while Section \ref{Section 4} establishes the strict positivity of $\Lambda$ near the boundary. Section \ref{Section 5} is devoted to the constant rank theorem for $\Lambda$, and in Section \ref{Section 6} we combine these results with the curvature flow deformation to prove Theorem \ref{Theorem-main theorem}. Finally, Section \ref{Section 7} shows that the exponent $1/2$ cannot be improved.

\section{Preliminaries}\label{Section 2}

In this section, we recall the elementary symmetric functions and the covariant derivative identities in hyperbolic space which will be used in the proof of the constant rank theorem. We begin with the definition of the $k$-th elementary symmetric function. For $1\leq k\leq n$ and $\lambda=(\lambda_1,\lambda_2,\cdots,\lambda_n)\in\mathbb{R}^n$, define $$\sigma_k(\lambda)=\sum_{1\leq i_1<i_2<\cdots<i_k\leq n}\lambda_{i_1}\lambda_{i_2}\cdots\lambda_{i_k}.$$ We also adopt the convention that $\sigma_0(\lambda)=1$ and $\sigma_{-1}(\lambda)=\sigma_{n+1}(\lambda)=0$. The definition of $\sigma_k$ for vectors in $\mathbb{R}^n$ can be extended to $n\times n$ symmetric matrices in a natural way by letting $\sigma_k(W)=\sigma_k(\lambda(W))$, where $\lambda(W)=(\lambda_1(W),\lambda_2(W),\cdots,\lambda_n(W))$ are the eigenvalues of the symmetric matrix $W$. We denote by $\sigma_{k-1}(\lambda|i)$ the $(k-1)$-th elementary symmetric function obtained by setting $\lambda_i=0$, and by $\sigma_{k-2}(\lambda|ij)$ $(i\neq j)$ the $(k-2)$-th elementary symmetric function obtained by setting $\lambda_i=\lambda_j=0$. When $W$ is diagonal, we use the same notation $\sigma_{k-1}(W|i)$ and $\sigma_{k-2}(W|ij)$ for the corresponding symmetric functions of its diagonal entries.

The first and second derivatives of the elementary symmetric functions are given in the following lemma.
\begin{lemma}[Proposition 2.2 in \cite{Guan-Ma2003}]\label{Lemma-sigma_k}
Suppose that $W=(w_{\alpha\beta})$ is diagonal. For $k=1,2,\cdots,n$, we have
\begin{equation}\label{Equality-sigma_{k-1}}
\frac{\partial\sigma_k(W)}{\partial w_{\alpha\beta}}=\begin{cases}
\sigma_{k-1}(W|\alpha) & \mathrm{if}\ \alpha=\beta,\\
0 & \mathrm{if}\ \alpha\neq\beta,
\end{cases}
\end{equation}
and $$\frac{\partial^2\sigma_k(W)}{\partial w_{\alpha\beta}\partial w_{\gamma\delta}}=\begin{cases}
\sigma_{k-2}(W|\alpha\gamma) & \mathrm{if}\ \alpha=\beta,\gamma=\delta,\alpha\neq\gamma,\\
-\sigma_{k-2}(W|\alpha\gamma) & \mathrm{if}\ \alpha=\delta,\beta=\gamma,\alpha\neq\gamma,\\
0 & \mathrm{otherwise}.
\end{cases}$$
\end{lemma}

We next record the differential identities in hyperbolic space $\mathbb{H}^n$ of constant sectional curvature $-1$. We use the curvature convention $$R_{ijkl}=-(\delta_{ik}\delta_{jl}-\delta_{il}\delta_{jk})$$ in a local orthonormal frame. For a smooth function $u$, subscripts denote covariant derivatives. In normal coordinates at a fixed point, the Ricci commutation identities give, for $i\neq j$,
\begin{equation}\label{Formula-commutation formula-3rd}
u_{jii}=u_{iij}-u_j
\end{equation}
and $$u_{ii\alpha\alpha}=u_{\alpha\alpha ii}-2u_{ii}+2u_{\alpha\alpha}.$$
Consequently,
\begin{equation}\label{Formula-commutation formula-4th}
\sum_{\alpha=1}^nu_{ii\alpha\alpha}=(\Delta u)_{ii}-2nu_{ii}+2\Delta u.
\end{equation}

For later use, we also recall the radial formulas in geodesic polar coordinates centered at a point $o\in\mathbb{H}^n$. The metric is $$g=\dd r^2+\sinh^2r\dd z^2,$$where $\dd z^2$ denotes the standard metric on $\mathbb{S}^{n-1}$. If $u=u(r)$ is radial, the Laplacian is
\begin{equation}\label{Equation-radial Laplacian}
\Delta u=u''+(n-1)\coth r u',
\end{equation}
and its Hessian is given by
\begin{equation}\label{Equality-radial Hessian}
\nabla^2u=u''\dd r^2+u'\coth r(g-\dd r^2).
\end{equation}

\section{The geodesic ball case and a uniform estimate}\label{Section 3}

In this section, we first prove that the tensor $\Lambda$ defined in \eqref{Tensor} is positive definite on geodesic balls of sufficiently small radius and then use the radial solution as a barrier to obtain a uniform estimate for the torsion function on general domains.

Let $B_R\subset\mathbb{H}^n$ be a geodesic ball of radius $R$, and let $u_B$ denote its torsion function, which satisfies $$\begin{cases}
\Delta u_B=-2 & \mathrm{in}\ B_R\\
\hspace{0.3cm} u_B=0 & \mathrm{on}\ \partial B_R.
\end{cases}$$ By rotational invariance and uniqueness, $u_B$ depends only on the radial variable $r$. Hence, by the radial Laplacian formula \eqref{Equation-radial Laplacian}, $u_B$ satisfies
\begin{equation}\label{Equation-torsion-radial-u}
u_B''+(n-1)\coth ru_B'=-2\quad\mathrm{in}\ (0,R)
\end{equation}
with $u_B'(0)=0$ and $u_B(R)=0.$ Equivalently, $$(\sinh^{n-1}ru_B')'=-2\sinh^{n-1}r,$$ and hence $$u_B'(r)=-2(\sinh r)^{1-n}\int_0^r\sinh^{n-1}s\dd s<0$$ for $r\in(0,R)$. Set $v_B=-\sqrt{u_B}$ and denote $\varphi=v_B'$. Since $u_B'<0$, we have $\varphi=-\frac{u_B'}{2u_B^{1/2}}>0$ in $(0,R)$, so that $|\nabla v_B|=\varphi$, while equation \eqref{Equation-torsion-v} then becomes
\begin{equation}\label{Equation-torsion-radial}
\varphi'=-(n-1)\coth r\varphi-\frac{1+\varphi^2}{v_B}.
\end{equation}

We now verify the positivity of the tensor \eqref{Tensor} for $v_B$. The following proposition also provides an explicit restriction on the radius of the ball.
\begin{proposition}\label{Proposition-ball case}
Define
\begin{equation}
\rho_0(n)=\mathrm{arccoth}\left(\frac{1+\sqrt{5+4/(n-1)}}{2}\right).
\end{equation}
If $0<R<\rho_0(n)$, then we have $$\nabla^2v_B-|\nabla v_B|g>0\quad \mathrm{in}\ B_R.$$
\end{proposition}
\begin{proof}
By \eqref{Equality-radial Hessian}, the radial eigenvalue of $\nabla^2v_B-|\nabla v_B|g$ is $\varphi'-\varphi$, while its tangential eigenvalue is $(\coth r-1)\varphi$ with multiplicity $n-1$. Since $\varphi>0$ and $\coth r>1$, it remains to prove that $$\varphi'-\varphi>0.$$

Let $h=\varphi'-\varphi$. At the center, since $\varphi(0)=0$, taking the limit $r\rightarrow 0$ in \eqref{Equation-torsion-radial} and applying L'H\^{o}pital's rule yield $$h(0)=\varphi'(0)=-\frac{1}{nv_B(0)}>0.$$ Arguing by contradiction, suppose that $h$ vanishes somewhere in $(0,R)$. Denoting its first zero by $r_1$, one has $$h(r)>0\quad\mathrm{for}\ 0\leq r<r_1,\quad h(r_1)=0,\quad h'(r_1)\leq 0,$$ while $h(r_1)=0$ means that $\varphi'(r_1)=\varphi(r_1)$, and \eqref{Equation-torsion-radial} at $r_1$ takes the form
\begin{equation}\label{Equation-torsion-radial-r_1}
(1+(n-1)\coth r_1)\varphi=-\frac{1+\varphi^2}{v_B}.
\end{equation}
Differentiating \eqref{Equation-torsion-radial} gives $$\varphi''=(n-1)\mathrm{csch}^2r\varphi-(n-1)\coth r\varphi'-\frac{2}{v_B}\varphi\varphi'+\frac{\varphi(1+\varphi^2)}{v_B^2},$$ which, together with \eqref{Equation-torsion-radial-r_1} and $\varphi'(r_1)=\varphi(r_1)$, yields at $r_1$
\begin{align*}
h'
=&~\varphi''-\varphi'\\
=&~(n-1)\mathrm{csch}^2r_1\varphi-(n-1)\coth r_1\varphi-\frac{2}{v_B}\varphi^2+\frac{\varphi(1+\varphi^2)}{v_B^2}-\varphi\\
=&~\varphi\bigg((n-1)\mathrm{csch}^2r_1-(1+(n-1)\coth r_1)\\
&\hspace{0.3cm}+\frac{2(1+(n-1)\coth r_1)\varphi^2}{1+\varphi^2}+\frac{(1+(n-1)\coth r_1)^2\varphi^2}{1+\varphi^2}\bigg)\\
\geq&~\varphi\left((n-1)\mathrm{csch}^2r_1-1-(n-1)\coth r_1\right).
\end{align*}
Moreover, in view of the identity $\mathrm{csch}^2r_1=\coth^2r_1-1$, we have $$h'\geq(n-1)\varphi\left(\coth^2r_1-\coth r_1-\frac{n}{n-1}\right).$$ The right-hand side is positive whenever $$\coth r_1>\frac{1+\sqrt{5+4/(n-1)}}{2}.$$ Since $$\rho_0(n)=\mathrm{arccoth}\left(\frac{1+\sqrt{5+4/(n-1)}}{2}\right),$$ for $R<\rho_0(n)$, it follows from $r_1<R<\rho_0(n)$ and the strict decrease of $\coth r$ that $$\coth r_1>\coth\rho_0(n).$$
This yields $h'(r_1)>0$, whereas $h'(r_1)\leq 0$, a contradiction. Hence $h>0$ on $[0,R)$, and the proposition follows.
\end{proof}

We next derive the estimate required in the constant rank theorem.
\begin{proposition}\label{Proposition-uniform estimate}
Let $n\geq 2$, and let $\Omega\subset\mathbb{H}^n$ be contained in a geodesic ball $B_R$ of radius $R>0$. If $u_\Omega$ is the torsion function of $\Omega$ and $v_\Omega=-\sqrt{u_\Omega}$, then $$-\frac{R}{\sqrt{n}}<v_\Omega<0\quad\mathrm{in}\ \Omega.$$ In particular, if $R\leq 1/(2\sqrt{n})$, then
\begin{equation}\label{Uniform estimate}
-\frac{1}{2n}<v_\Omega<0\quad\mathrm{in}\ \Omega.
\end{equation}
\end{proposition}
\begin{proof}
Setting $A(r)=\sinh^{n-1}r$ and multiplying equation \eqref{Equation-torsion-radial-u} by $A(r)$ yield $$(A(r)u_B'(r))'=-2A(r).$$ Since $u_B'(0)=0$, integration over $[0,r]$ gives $$u_B'(r)=-\frac{2}{A(r)}\int_0^rA(s)\dd s.$$ Using $u_B(R)=0$, we further obtain $$u_B(r)=2\int_r^R\frac{1}{A(t)}\left(\int_0^tA(s)\dd s\right)\dd t.$$ In particular, $$u_B(0)=2\int_0^R\frac{1}{\sinh^{n-1}t}\left(\int_0^t\sinh^{n-1}s\dd s\right)\dd t.$$ Notice that $\sinh\tau/\tau$ is increasing on $(0,\infty)$, and hence $\sinh s/\sinh t<s/t$ for $0<s<t$. It follows that
\begin{align*}
\frac{1}{\sinh^{n-1}t}\int_0^t\sinh^{n-1}s\dd s
=&~\int_0^t\left(\frac{\sinh s}{\sinh t}\right)^{n-1}\dd s\\
<&~\int_0^t\left(\frac{s}{t}\right)^{n-1}\dd s=\frac{t}{n}.
\end{align*}
Therefore, $u_B(0)<R^2/n$. From $u_B'<0$ in $(0,R)$, we deduce $0<u_B<R^2/n$ in $B_R$, which implies $$-\frac{R}{\sqrt{n}}<v_B<0\quad\mathrm{in}\ B_R.$$ Moreover, if $\Omega\subset B_R$, the comparison principle yields $u_\Omega\leq u_B$ in $\Omega$, and hence $$-\frac{R}{\sqrt{n}}<v_\Omega<0\quad\mathrm{in}\ \Omega,$$ which proves the proposition.
\end{proof}

Finally, a direct computation shows that $$\frac{1}{2\sqrt{n}}<\rho_0(n)$$ for every $n\geq 2$. Therefore, the choice $$r_0(n)=\frac{1}{2\sqrt{n}}$$ guarantees both the strict positivity of $\Lambda$ on geodesic balls and the uniform estimate \eqref{Uniform estimate}.

\section{Boundary convexity estimate}\label{Section 4}

We next establish the boundary convexity estimate needed in the proof of the main theorem. The argument closely follows the proof of Wei and Xiao \cite[Lemma 4.1]{Wei-Xiao2025} and is included for completeness.
\begin{proposition}[Boundary convexity estimate]\label{Proposition-boundary estimate}
Let $\Omega$ be a bounded, smooth, strictly horo-convex domain in $\mathbb{H}^n$ for $n\geq 2$, and let $u$ solve \eqref{Equation-torsion problem}. For $v=-\sqrt{u}$, there is a small neighborhood $\mathcal{N}$ of $\partial\Omega$ such that $$\nabla^2v-|\nabla v|g>0\quad \mathrm{in}\ \mathcal{N}\cap\Omega.$$
\end{proposition}
\begin{proof}
Let $\nu$ denote the unit outward normal to $\partial\Omega$. By the maximum principle and the Hopf lemma, we have $u>0$ in $\Omega$ and $\nabla_\nu u<0$ on $\partial\Omega$. Since $u=0$ on $\partial\Omega$, all its tangential derivatives vanish there, and hence $$\nabla u=(\nabla_\nu u)\nu=-|\nabla u|\nu\quad\mathrm{on}\ \partial\Omega.$$ For any $p\in\partial\Omega$ and $e\in T_p\partial\Omega$, we have $$\nabla^2u(e,e)=e(e(u))-(\nabla_ee)u=-\Braket{\nabla_ee,\nabla u}=\mathrm{II}(e,e)\nabla_\nu u=-\mathrm{II}(e,e)|\nabla u|,$$ where $\mathrm{II}$ is the second fundamental form of $\partial\Omega$ and $\Braket{\cdot,\cdot}$ denotes the inner product induced by the hyperbolic metric $g$.

Direct computation in $\Omega$ yields $$\nabla v=-\frac{\nabla u}{2u^{1/2}},\quad |\nabla v|=\frac{|\nabla u|}{2u^{1/2}}$$ and $$\nabla^2v=\frac{\nabla u\otimes\nabla u}{4u^{3/2}}-\frac{\nabla^2u}{2u^{1/2}}.$$ Therefore, for any vector $e$,
\begin{equation}\label{Boundary convexity estimate-general}
\nabla^2v(e,e)-|\nabla v||e|^2=\frac{1}{2u^{1/2}}\left(\frac{|\nabla_eu|^2}{2u}-\nabla^2u(e,e)-|\nabla u||e|^2\right).
\end{equation}
When $e\in T_p\partial\Omega$ is a unit tangential vector, strict horo-convexity implies
\begin{equation}\label{Boundary convexity estimate-tangential}
-\nabla^2u(e,e)-|\nabla u||e|^2=|\nabla u|(\mathrm{II}(e,e)-|e|^2)>0.
\end{equation}

Near $\partial\Omega$, write a general unit vector as $e=e^\top+e^\bot$, where $e^\bot=\Braket{e,\nabla u/|\nabla u|}\nabla u/|\nabla u|$. Then $$|\nabla_eu|^2=\Braket{e,\nabla u}^2=\left\langle e,|\nabla u|\frac{\nabla u}{|\nabla u|}\right\rangle^2=|\nabla u|^2\left\langle e,\frac{\nabla u}{|\nabla u|}\right\rangle^2=|\nabla u|^2|e^\bot|^2,$$ and substituting into \eqref{Boundary convexity estimate-general}, we obtain $$\nabla^2v(e,e)-|\nabla v||e|^2
=\frac{1}{2u^{1/2}}\left(\frac{|\nabla u|^2|e^\bot|^2}{2u}-\nabla^2u(e,e)-|\nabla u||e|^2\right).$$ By \eqref{Boundary convexity estimate-tangential}, compactness and continuity, there exists $\eta>0$ such that $\nabla^2v(e,e)-|\nabla v||e|^2>0$ near $\partial\Omega$ for every unit vector $e$ satisfying $|e^\bot|\leq\eta$. Since $u\rightarrow 0$, $|\nabla u|$ is bounded away from zero, and $\nabla^2u$ is bounded near $\partial\Omega$, we have $$\nabla^2v(e,e)-|\nabla v||e|^2\rightarrow+\infty$$ uniformly for $|e^\bot|\geq\eta$ as $x\rightarrow\partial\Omega$. Hence the tensor $\nabla^2v-|\nabla v|g$ is positive definite in a small neighborhood of $\partial\Omega$, which completes the proof.
\end{proof}

\section{The constant rank theorem}\label{Section 5}

We now turn to a key result of the paper, namely the constant rank theorem for the tensor $\Lambda=\nabla^2v-|\nabla v|g$ defined in \eqref{Tensor}, which states that $\Lambda$ has constant rank whenever it is nonnegative definite and $v$ satisfies the uniform estimate \eqref{Uniform estimate}.
\begin{proposition}[Constant rank theorem]\label{Proposition-constant rank theorem}
Let $\Omega$ be a bounded, smooth domain in $\mathbb{H}^n$ for $n\geq 2$, and let $v\in C^\infty(\Omega)$ be a solution of
\begin{equation}\label{Equation-v}
\Delta v=-\frac{1+|\nabla v|^2}{v}
\end{equation}
satisfying
\begin{equation}\label{Assumption}
-\frac{1}{2n}<v<0\quad\mathrm{in}\ \Omega.
\end{equation}
If $\Lambda=\nabla^2v-|\nabla v|g\geq 0$ in $\Omega$, then $\Lambda$ has constant rank in $\Omega$.
\end{proposition}
\begin{proof}
Let $l$ be the minimal rank of $\Lambda$ in $\Omega$, $l=0,1,\cdots,n$, and suppose that this rank is attained at $x_0$. We define the auxiliary function $$\phi(x)=\sigma_{l+1}(\Lambda(x)).$$ Following the standard argument, it suffices to prove that there exists a neighborhood $\mathcal{O}$ of $x_0$ and a positive constant $C$, independent of $\phi$, such that 
\begin{equation}\label{Inequality-Goal}
\Delta\phi\leq C(\phi+|\nabla\phi|)\quad\mathrm{in}\ \mathcal{O}.
\end{equation}

For any $x\in\Omega$, let $\lambda_1\geq\lambda_2\geq\cdots\geq\lambda_n$ be the eigenvalues of $\Lambda$ at $x$. There is a positive constant $c>0$ such that $\lambda_1\geq\lambda_2\geq\cdots\geq\lambda_l\geq c$. Let $G=\{1,2,\cdots,l\}$ and $B=\{l+1,l+2,\cdots,n\}$ be the “good" and “bad" sets of indices, respectively. If there is no confusion, we also denote the “good" and “bad" eigenvalues of $\Lambda$ by $G$ and $B$, respectively.

Following the notation in \cite{Caffarelli-Friedman1985}, for two functions $h(x)$ and $k(x)$ defined in $\mathcal{O}$, we say that $h(x)\lesssim k(x)$ if there exists a positive constant $C_1$ such that $$h(x)-k(x)\leq C_1(\phi(x)+|\nabla\phi(x)|).$$ We also write $h(x)\sim k(x)$ if $h(x)\lesssim k(x)$ and $k(x)\lesssim h(x)$. Next, we write $h\lesssim k$ if the above inequality holds for any $x\in\mathcal{O}$, with constant $C_1$ independent of $x$. Finally, $h\sim k$ means $h\lesssim k$ and $k\lesssim h$.

For each fixed $x\in\mathcal{O}$, by choosing a local orthonormal frame $\{e_1,e_2,\cdots,e_n\}$, we may assume that $\Lambda$ is diagonal at $x$ and $\Lambda_{ii}=\lambda_i$ for any $i=1,2,\cdots,n$. The following calculations will be carried out at a fixed point $x$. When we use the relation $\lesssim$, all constants are under control.

The case $l=n$ is immediate, since $\Lambda$ is then of full rank in $\Omega$. Throughout the proof, set $w=|\nabla v|$. When $l=0$, the assumption \eqref{Assumption},  together with \eqref{Equation-v}, gives
\begin{equation}\label{sigma_1(Lambda)}
\begin{aligned}
\sigma_1(\Lambda)
=&~\Delta v-nw=-\frac{1}{v}-\frac{1}{v}w^2-nw\\
=&-\frac{1}{v}(w^2+nvw+1)\\
=&-\frac{1}{v}\left(\left(w+\frac{n}{2}v\right)^2+1-\frac{n^2}{4}v^2\right)\\
>&~0.
\end{aligned}
\end{equation}
Therefore, in what follows, we only need to consider the case $l=1,2,\cdots,n-1$.

We next verify that $w(x_0)\neq 0$. Suppose to the contrary that $w(x_0)=0$. Since $\Lambda(x_0)$ is degenerate and \eqref{sigma_1(Lambda)} holds, we may choose orthonormal eigenvectors $e_1$ and $e_n$ such that $$v_{11}(x_0)>0\quad\mathrm{and}\quad v_{nn}(x_0)=0.$$ Extend these vectors by parallel translation near $x_0$ and consider $$\widetilde{\Lambda}_{ij}=v_{ij}-\Braket{\nabla v,e_1}\delta_{ij}.$$ Since $|\nabla v|\geq\Braket{\nabla v,e_1}$, we have $\widetilde{\Lambda}\geq\Lambda\geq 0$ near $x_0$, and $\widetilde{\Lambda}_{nn}(x_0)=0$ therefore yields $$0=\widetilde{\Lambda}_{nn1}(x_0)=v_{nn1}(x_0)-v_{11}(x_0).$$ On the other hand, the nonnegative definiteness of $\nabla^2v=\Lambda+wg$, together with $v_{nn}(x_0)=0$, implies that $v_{nn}$ attains a local minimum at $x_0$, and therefore $v_{nn1}(x_0)=0$, which contradicts $v_{11}(x_0)>0$. By decreasing the size of $\mathcal{O}$ if necessary, we may assume that $$w>0\quad\mathrm{in}\ \mathcal{O}.$$

We now return to the auxiliary function $\phi$. Since $\Lambda$ is diagonal, we have that $$0\sim\phi\sim\sigma_l(G)\sum_{i\in B}\Lambda_{ii}\sim\sum_{i\in B}\Lambda_{ii},$$ and
\begin{equation}\label{Relation-Lambda_{ii}}
\Lambda_{ii}\sim 0,\quad i\in B.
\end{equation}
Equivalently,
\begin{equation}\label{Relation-v_{ii}=w}
v_{ii}\sim w,\quad i\in B.
\end{equation}
It follows from \eqref{Relation-Lambda_{ii}} that
\begin{equation}\label{Relation-sigma_l}
\sigma_l(\lambda|i)\sim\begin{cases}
0 & \mathrm{if}\ i\in G,\\
\sigma_l(G) & \mathrm{if}\ i\in B
\end{cases}
\end{equation}
holds. For $i\neq j$, we have
\begin{equation}\label{Relation-sigma_{l-1}}
\sigma_{l-1}(\lambda|ij)\sim
\begin{cases}
0 & \mathrm{if}\ i\in G, j\in G,\\
\sigma_{l-1}(G|i) & \mathrm{if}\ i\in G, j\in B,\\
\sigma_{l-1}(G|j) & \mathrm{if}\ i\in B, j\in G,\\
\sigma_{l-1}(G) & \mathrm{if}\ i\in B, j\in B.
\end{cases}
\end{equation}
Differentiating $\phi(x)$ and applying \eqref{Equality-sigma_{k-1}} and \eqref{Relation-sigma_l}, we obtain
\begin{align*}
0\sim\phi_\alpha
=&\sum_{i,j}\frac{\partial\sigma_{l+1}}{\partial \Lambda_{ij}}\Lambda_{ij\alpha}=\sum_i\sigma_l(\lambda|i)\Lambda_{ii\alpha}\\
=&\sum_{i\in G}\sigma_l(\lambda|i)\Lambda_{ii\alpha}+\sum_{i\in B}\sigma_l(\lambda|i)\Lambda_{ii\alpha}\\
\sim&~\sigma_l(G)\sum_{i\in B}\Lambda_{ii\alpha},
\end{align*}
which implies
\begin{equation}\label{Relation-Lambda_{iialpha}}
\sum_{i\in B}\Lambda_{ii\alpha}\sim 0
\end{equation}
and
\begin{equation}\label{Relation-v_{iialpha}=(n-l)w_alpha}
\sum_{i\in B}v_{ii\alpha}\sim(n-l)w_\alpha.
\end{equation}

Combining \eqref{Relation-sigma_l}-\eqref{Relation-Lambda_{iialpha}} with Lemma \ref{Lemma-sigma_k}, for $\alpha=1,2,\cdots,n$ we calculate the second derivatives of $\phi(x)$,
\begin{align*}
\phi_{\alpha\alpha}
=&~\sum_{i,j}\frac{\partial\sigma_{l+1}}{\partial \Lambda_{ij}}\Lambda_{ij\alpha\alpha}+\sum_{i,j,p,q}\frac{\partial^2\sigma_{l+1}}{\partial \Lambda_{ij}\partial \Lambda_{pq}}\Lambda_{ij\alpha}\Lambda_{pq\alpha}\\
=&~\sum_i\sigma_l(\lambda|i)\Lambda_{ii\alpha\alpha}+\sum_{i\neq j}\sigma_{l-1}(\lambda|ij)\Lambda_{ii\alpha}\Lambda_{jj\alpha}-\sum_{i\neq j}\sigma_{l-1}(\lambda|ij)\Lambda_{ij\alpha}^2\\
\sim&~\sigma_l(G)\sum_{i\in B}\Lambda_{ii\alpha\alpha}+\bigg(\sum_{\substack{i\in G\\ j\in B}}+\sum_{\substack{i\in B\\ j\in G}}+\sum_{\substack{i,j\in G\\ i\neq j}}+\sum_{\substack{i,j\in B\\ i\neq j}}\bigg)\sigma_{l-1}(\lambda|ij)\Lambda_{ii\alpha}\Lambda_{jj\alpha}\\
&-\bigg(\sum_{\substack{i\in G\\ j\in B}}+\sum_{\substack{i\in B\\ j\in G}}+\sum_{\substack{i,j\in G\\ i\neq j}}+\sum_{\substack{i,j\in B\\ i\neq j}}\bigg)\sigma_{l-1}(\lambda|ij)\Lambda_{ij\alpha}^2\\
\sim&~\sigma_l(G)\sum_{i\in B}\Lambda_{ii\alpha\alpha}+\sum_{i\in G}\sigma_{l-1}(G|i)\Lambda_{ii\alpha}\sum_{j\in B}\Lambda_{jj\alpha}+\sum_{i\in B}\Lambda_{ii\alpha}\sum_{j\in G}\sigma_{l-1}(G|j)\Lambda_{jj\alpha}\\
&+\sigma_{l-1}(G)\sum_{\substack{i,j\in B\\ i\neq j}}\Lambda_{ii\alpha}\Lambda_{jj\alpha}-2\sum_{\substack{i\in G\\ j\in B}}\sigma_{l-1}(G|i)\Lambda_{ij\alpha}^2-\sigma_{l-1}(G)\sum_{\substack{i,j\in B\\ i\neq j}}\Lambda_{ij\alpha}^2\\
\sim&~\sigma_l(G)\sum_{i\in B}\Lambda_{ii\alpha\alpha}+\sigma_{l-1}(G)\sum_{\substack{i,j\in B\\ i\neq j}}\Lambda_{ii\alpha}\Lambda_{jj\alpha}-2\sum_{\substack{i\in G\\ j\in B}}\sigma_{l-1}(G|i)\Lambda_{ij\alpha}^2-\sigma_{l-1}(G)\sum_{\substack{i,j\in B\\ i\neq j}}\Lambda_{ij\alpha}^2.
\end{align*}
From condition \eqref{Relation-Lambda_{iialpha}}, we have $-\Lambda_{ii\alpha}\sim\sum_{\substack{j\in B\\ j\neq i}}\Lambda_{jj\alpha},$ and hence $$\sigma_{l-1}(G)\sum_{\substack{i,j\in B\\ i\neq j}}\Lambda_{ii\alpha}\Lambda_{jj\alpha}\sim-\sigma_{l-1}(G)\sum_{i\in B}\Lambda_{ii\alpha}^2.$$ Since $\sigma_{l-1}(G|i)=\sigma_l(G)/\lambda_i$ for $i\in G$, we conclude that $$\phi_{\alpha\alpha}\sim\sigma_l(G)\sum_{i\in B}\Lambda_{ii\alpha\alpha}-2\sigma_l(G)\sum_{\substack{i\in G\\ j\in B}}\frac{\Lambda_{ij\alpha}^2}{\lambda_i}-\sigma_{l-1}(G)\sum_{i,j\in B}\Lambda_{ij\alpha}^2.$$ Summing over $\alpha$ yields
\begin{equation}\label{Equality-F1}
\Delta\phi\sim\sigma_l(G)\sum_\alpha\sum_{i\in B}\Lambda_{ii\alpha\alpha}-2\sigma_l(G)\sum_\alpha\sum_{\substack{i\in G\\ j\in B}}\frac{\Lambda_{ij\alpha}^2}{\lambda_i}-\sigma_{l-1}(G)\sum_\alpha\sum_{i,j\in B}\Lambda_{ij\alpha}^2.
\end{equation}

We now calculate the first and second derivatives of $w$. From $w_i=(|\nabla v|)_i=\sum_jv_jv_{ji}/w$, set $a_i=v_i/w$, so that
\begin{equation}\label{Relation-a_i}
\sum_i a_i^2=1
\end{equation}
and
\begin{equation}\label{Relation-w_i=a_iv_{ii}}
w_i=a_iv_{ii}
\end{equation}
at $x$. By the commutation formula \eqref{Formula-commutation formula-3rd} on hyperbolic space, we obtain $$w_{ii}=-w^{-3}\sum_{j,k}v_jv_kv_{ji}v_{ki}+w^{-1}\sum_jv_{ji}^2+w^{-1}\sum_{j\neq i}v_j(v_{iij}-v_j)+w^{-1}v_iv_{iii},$$ which implies, at $x$,
\begin{equation}\label{Equality-w_{ii}}
w_{ii}=-w^{-1}a_i^2v_{ii}^2+w^{-1}v_{ii}^2+\sum_{j\neq i}a_j(v_{iij}-v_j)+a_iv_{iii}.
\end{equation}

Denote the three terms on the right-hand side of \eqref{Equality-F1} by $\circled{1}$, $\circled{2}$, and $\circled{3}$, respectively, and estimate them separately.

$\bullet$ \circled{1}. Note that $\Lambda_{ii\alpha\alpha}=v_{ii\alpha\alpha}-w_{\alpha\alpha}$. For fixed $i\in B$, the commutation formula \eqref{Formula-commutation formula-4th} and equation \eqref{Equation-v} give
\begin{align*}
\sum_\alpha v_{ii\alpha\alpha}
=&~(\Delta v)_{ii}-2nv_{ii}+2\Delta v\\
=&-2v^{-3}v_i^2(1+w^2)+v^{-2}v_{ii}(1+w^2)+4v^{-2}v_iww_i\\
&-2v^{-1}w_i^2-2v^{-1}ww_{ii}-2nv_{ii}-2v^{-1}(1+w^2).
\end{align*}
Using \eqref{Relation-v_{ii}=w}, \eqref{Relation-a_i}-\eqref{Equality-w_{ii}}, and $v_i=a_iw$, we obtain
\begin{equation}\label{Equality-v_{iialphaalpha}}
\begin{aligned}
\sum_\alpha v_{ii\alpha\alpha}
\sim&-2v^{-3}a_i^2w^2(1+w^2)+v^{-2}w(1+w^2)+4v^{-2}a_i^2w^3-2v^{-1}a_i^2w^2\\
&-2v^{-1}w\sum_\alpha a_\alpha v_{ii\alpha}-2nw-2v^{-1}(1+w^2).
\end{aligned}
\end{equation}
On the other hand, in view of \eqref{Equality-w_{ii}}, a straightforward calculation yields
\begin{equation}\label{Equality-w_{alphaalpha}}
\begin{aligned}
\sum_\alpha w_{\alpha\alpha}
=&~w^{-1}\sum_\alpha(1-a_\alpha^2)v_{\alpha\alpha}^2+\sum_\alpha a_\alpha(\Delta v)_\alpha-(n-1)w\\
=&~w^{-1}\sum_\alpha(1-a_\alpha^2)v_{\alpha\alpha}^2+v^{-2}w(1+w^2)-2v^{-1}w\sum_\alpha a_\alpha w_\alpha-(n-1)w.
\end{aligned}
\end{equation}
Subtracting \eqref{Equality-w_{alphaalpha}} from \eqref{Equality-v_{iialphaalpha}} and applying $v_{ii\alpha}=\Lambda_{ii\alpha}+w_\alpha$, we have
\begin{equation}\label{Equality-Lambda_{iialphaalpha}}
\begin{aligned}
\sum_\alpha\Lambda_{ii\alpha\alpha}
\sim&-2v^{-3}a_i^2w^2(1+w^2)+4v^{-2}a_i^2w^3-2v^{-1}a_i^2w^2-2v^{-1}w\sum_\alpha a_\alpha\Lambda_{ii\alpha}\\
&-(n+1)w-2v^{-1}(1+w^2)-w^{-1}\sum_\alpha(1-a_\alpha^2)v_{\alpha\alpha}^2.
\end{aligned}
\end{equation}
Since $v_{\alpha\alpha}=\lambda_\alpha+w$, it follows from \eqref{Relation-Lambda_{ii}} and \eqref{Relation-a_i} that the last term can be written as 
\begin{equation}
w^{-1}\sum_\alpha(1-a_\alpha^2)v_{\alpha\alpha}^2
\sim w^{-1}\sum_{\alpha\in G}(1-a_\alpha^2)\lambda_\alpha^2+2\sum_{\alpha\in G}\lambda_\alpha-2\sum_{\alpha\in G}a_\alpha^2\lambda_\alpha+(n-1)w.
\end{equation}
Moreover, by \eqref{Equation-v} and \eqref{Relation-Lambda_{ii}},
\begin{equation}\label{Equality-sum of lambda_i}
\sum_{\alpha\in G}\lambda_\alpha\sim-\frac{1+w^2+nvw}{v}.
\end{equation}
Combining \eqref{Equality-Lambda_{iialphaalpha}}-\eqref{Equality-sum of lambda_i}, we arrive at
\begin{align*}
\sum_\alpha\Lambda_{ii\alpha\alpha}
\sim&-2v^{-3}a_i^2w^2(1+w^2-2vw+v^2)-2v^{-1}w\sum_\alpha a_\alpha\Lambda_{ii\alpha}\\
&-w^{-1}\sum_{\alpha\in G}(1-a_\alpha^2)\lambda_\alpha^2+2\sum_{\alpha\in G}a_\alpha^2\lambda_\alpha.
\end{align*}
Therefore, the term \circled{1} becomes
\begin{align*}
\circled{1}
=&~\sigma_l(G)\sum_\alpha\sum_{i\in B}\Lambda_{ii\alpha\alpha}\\
\sim&-2v^{-3}w^2(1+w^2-2vw+v^2)\sigma_l(G)\sum_{i\in B}a_i^2\\
&-2v^{-1}w\sigma_l(G)\sum_\alpha a_\alpha\sum_{i\in B}\Lambda_{ii\alpha}-w^{-1}\sigma_l(G)\sum_{i\in B}\sum_{\alpha\in G}(1-a_\alpha^2)\lambda_\alpha^2\\
&+2(n-l)\sigma_l(G)\sum_{\alpha\in G}a_\alpha^2\lambda_\alpha.
\end{align*}
For $i\in B$ and $\alpha\in G$, $1-a_\alpha^2=\sum_{\beta\neq\alpha}a_\beta^2\geq a_i^2$. Combining this with \eqref{Relation-Lambda_{iialpha}} yields the estimate
\begin{equation}\label{Estimate-1}
\begin{aligned}
\circled{1}
\lesssim&-2v^{-3}w^2(1+w^2-2vw+v^2)\sigma_l(G)\sum_{i\in B}a_i^2-w^{-1}\sigma_l(G)\sum_{\alpha\in G}\lambda_\alpha^2\sum_{i\in B}a_i^2\\
&+2(n-l)\sigma_l(G)\sum_{\alpha\in G}a_\alpha^2\lambda_\alpha.
\end{aligned}
\end{equation}

$\bullet$ \circled{2}. Observe that $$-2\sigma_l(G)\sum_\alpha\sum_{\substack{i\in G\\ j\in B}}\frac{\Lambda_{ij\alpha}^2}{\lambda_i}\leq-2\sigma_l(G)\sum_{i\in G}\sum_{j\in B}\frac{\Lambda_{iji}^2}{\lambda_i}-2\sigma_l(G)\sum_{i\in G}\sum_{j\in B}\frac{\Lambda_{ijj}^2}{\lambda_i}.$$ For the first term on the right-hand side, fixing $j\in B$ and applying the Cauchy-Schwarz inequality gives $$-\sum_{i\in G}\frac{\Lambda_{iji}^2}{\lambda_i}\leq-\frac{(\sum_{i\in G}\Lambda_{iji})^2}{\sum_{i\in G}\lambda_i}.$$ By the commutation formula \eqref{Formula-commutation formula-3rd}, $$\sum_{i\in G}\Lambda_{iji}=\sum_{i\in G}v_{iji}=\sum_{i\in G}(v_{iij}-v_j).$$ Noting that $$\sum_{i\in G}v_{iij}+\sum_{i\in B}v_{iij}=(\Delta v)_j=v^{-2}v_j(1+w^2)-2v^{-1}ww_j$$ and using \eqref{Relation-v_{ii}=w}, \eqref{Relation-v_{iialpha}=(n-l)w_alpha}, \eqref{Relation-w_i=a_iv_{ii}}, we obtain
\begin{align*}
\sum_{i\in G}v_{iij}
\sim&~v^{-2}v_j(1+w^2)-2v^{-1}ww_j-(n-l)w_j\\
\sim&~v^{-2}a_jw(1+w^2)-2v^{-1}a_jw^2-(n-l)a_jw.
\end{align*}
Thus, $$\sum_{i\in G}\Lambda_{iji}\sim a_jw\big(v^{-2}(1+w^2)-2v^{-1}w-n\big).$$ Furthermore,
\begin{equation}\label{Estimate-2_1}
-\sum_{i\in G}\frac{\Lambda_{iji}^2}{\lambda_i}\lesssim-\frac{1}{\sum_{i\in G}\lambda_i}a_j^2w^2\big(v^{-2}(1+w^2)-2v^{-1}w-n\big)^2.
\end{equation}

For the second term, fix $i\in G$. The Cauchy-Schwarz inequality yields $$-\sum_{j\in B}\Lambda_{ijj}^2\leq-\frac{1}{n-l}\big(\sum_{j\in B}\Lambda_{ijj}\big)^2.$$ The commutation formula \eqref{Formula-commutation formula-3rd}, together with \eqref{Relation-v_{iialpha}=(n-l)w_alpha} and \eqref{Relation-w_i=a_iv_{ii}}, then implies $$\sum_{j\in B}\Lambda_{ijj}=\sum_{j\in B}v_{ijj}=\sum_{j\in B}(v_{jji}-v_i)\sim(n-l)(w_i-v_i)=(n-l)a_i\lambda_i.$$ It follows that
\begin{equation}\label{Estimate-2_2}
-\sum_{j\in B}\frac{\Lambda_{ijj}^2}{\lambda_i}\lesssim-(n-l)a_i^2\lambda_i.
\end{equation}
Combining \eqref{Estimate-2_1} and \eqref{Estimate-2_2}, we obtain the estimate
\begin{equation}\label{Estimate-2}
\begin{aligned}
\circled{2}
\lesssim&-\frac{2}{\sum_{i\in G}\lambda_i}v^{-4}w^2(1+w^2-2vw-nv^2)^2\sigma_l(G)\sum_{j\in B}a_j^2\\
&-2(n-l)\sigma_l(G)\sum_{i\in G}a_i^2\lambda_i.
\end{aligned}
\end{equation}

$\bullet$ \circled{3}. The third term \circled{3} in \eqref{Equality-F1} is nonpositive and may be discarded.

Substituting the estimates \eqref{Estimate-1} and \eqref{Estimate-2} into \eqref{Equality-F1}, it follows that
\begin{equation}\label{Inequality-F1}
\begin{aligned}
\Delta\phi
\lesssim&~\sigma_l(G)\sum_{i\in B}a_i^2\bigg(-2v^{-3}w^2(1+w^2-2vw+v^2)-w^{-1}\sum_{\alpha\in G}\lambda_\alpha^2\\
&\hspace{2.4cm}-\frac{2}{\sum_{i\in G}\lambda_i}v^{-4}w^2(1+w^2-2vw-nv^2)^2\bigg).
\end{aligned}
\end{equation}
By exploiting the Cauchy-Schwarz inequality $$\sum_{\alpha\in G}\lambda_\alpha^2\geq\frac{1}{l}\big(\sum_{\alpha\in G}\lambda_\alpha\big)^2,$$ together with the previously established relation \eqref{Equality-sum of lambda_i}, the estimate \eqref{Inequality-F1} can be further estimated by
\begin{align*}
\Delta\phi
\lesssim&~\sigma_l(G)\sum_{i\in B}a_i^2\bigg(-\frac{2w^2(1+w^2-2vw+v^2)}{v^3}-\frac{(1+w^2+nvw)^2}{lv^2w}\\
&\hspace{2.4cm}+\frac{2w^2(1+w^2-2vw-nv^2)^2}{v^3(1+w^2+nvw)}\bigg).
\end{align*}

Set $$L=-\frac{2w^2(1+w^2-2vw+v^2)}{v^3}-\frac{(1+w^2+nvw)^2}{lv^2w}+\frac{2w^2(1+w^2-2vw-nv^2)^2}{v^3(1+w^2+nvw)}.$$ In order to establish \eqref{Inequality-Goal}, it remains to prove that $L\leq 0$. The combination of the first and third terms yields, after a direct computation,
\begin{align*}
L
=&~\frac{2w^2}{v^2(1+w^2+nvw)}\big(-(n+2)w(1+w^2)+v(-2n-1+3w^2)+3nv^2w+n^2v^3\big)\\
&-\frac{(1+w^2+nvw)^2}{lv^2w}.
\end{align*}
Now define $$M=1+w^2+nvw,\quad N=-(n+2)w(1+w^2)+v(-2n-1+3w^2)+3nv^2w+n^2v^3.$$ With this notation, $L$ takes the form $$L=\frac{2w^2N}{v^2M}-\frac{M^2}{lv^2w}.$$ Since $-1/(2n)<v<0$, we have $$M>1+w^2-\frac{w}{2}=\left(w-\frac{1}{4}\right)^2+\frac{15}{16}\geq\frac{15}{16}.$$ To ensure that $L\leq 0$, it is sufficient to verify that
\begin{equation}\label{Inequality-Goal'}
2lw^3N\leq M^3.
\end{equation}
The case $N\leq 0$ is immediate, so we only need to treat the case $N>0$. Notice that
\begin{align*}
N
\leq&-(n+2)w-(n+2)w^3-(2n+1)v+3nv^2w\\
\leq&~\frac{5}{4}-\left(n+2-\frac{3}{4n}\right)w-(n+2)w^3,
\end{align*}
which follows from $-1/(2n)<v<0$. Since $N>0$, we obtain $$\frac{5}{4}-\left(n+2-\frac{3}{4n}\right)w>0,$$ which is equivalent to $$w<\frac{5}{4c_n},\quad c_n=n+2-\frac{3}{4n}>n.$$ Meanwhile, we also know from the preceding estimate that $N\leq\frac{5}{4}$. Consequently, the conditions $l\leq n-1$ and $n\geq 2$ imply that
\begin{align*}
2lw^3N
\leq&~2(n-1)\left(\frac{5}{4c_n}\right)^3\frac{5}{4}\leq2\left(\frac{5}{4}\right)^4\frac{n-1}{n^3}\\
\leq&~2\left(\frac{5}{4}\right)^4\frac{1}{2^3}=\frac{625}{1024}<\left(\frac{15}{16}\right)^3\leq M^3.
\end{align*}
Therefore, \eqref{Inequality-Goal'} is proved.

Since $\phi\geq 0$ and $\phi(x_0)=0$, the strong maximum principle applied to \eqref{Inequality-Goal} yields $\phi\equiv 0\quad\mathrm{in}\ \mathcal{O}$, and the standard connectedness argument then shows that $\Lambda$ has rank $l$ throughout $\Omega$.

\end{proof}

\section{Proof of the main theorem}\label{Section 6}

We now complete the proof of the main theorem using the continuity method, combining the result for geodesic balls with the boundary convexity estimate and the constant rank theorem established in the previous sections. The key step is to deform the given domain to a geodesic ball while preserving strict horo-convexity. Throughout the deformation, all evolving domains remain inside a fixed ball of sufficiently small radius.

Let $M_0=\partial\Omega$, where $\Omega$ is a smooth strictly horo-convex domain containing a point $o$. The strict horo-convexity of $\Omega$ ensures that $M_0$ is an admissible initial hypersurface for the locally constrained curvature flow of Hu, Li, and Wei \cite{Hu-LW2022}, which we apply following Wei and Xiao \cite{Wei-Xiao2025}. The flow gives rise to a smooth family of strictly horo-convex hypersurfaces $\{M_t\}_{t\geq 0}$ with $M_t=\partial\Omega_t$ and $\Omega_0=\Omega$. Moreover, $M_t$ converges smoothly and exponentially to a geodesic sphere centered at $o$, so that $\Omega_t$ converges to the enclosed geodesic ball. The flow also preserves the containment property $\Omega_t\subset B_R(o)$ whenever $\Omega\subset B_R(o)$. Since the result of Hu, Li, and Wei \cite{Hu-LW2022} is established for $n\geq 3$, the case $n=2$ follows from the dimension reduction argument of Wei and Xiao. We refer to \cite[Sections 3 and 4]{Wei-Xiao2025} for the precise formulation of the flow and the properties used here.

\begin{proof}[Proof of Theorem \ref{Theorem-main theorem}]
Fix $o\in\Omega$, and observe that, with the choice $r_0(n)=1/(2\sqrt{n})$ made in Section \ref{Section 3}, the assumption $\mathrm{diam}(\Omega)\leq r_0(n)$ implies $\Omega\subset B_{r_0(n)}(o)$. Let $\{\Omega_t\}_{t\geq 0}$ be the deformation described above, denote by $u_t$ the torsion function of $\Omega_t$, and set $$v_t=-\sqrt{u_t},\quad \Lambda_t=\nabla^2v_t-|\nabla v_t|g.$$ Since the deformation preserves the containment $\Omega_t\subset B_{r_0(n)}(o)$, Proposition \ref{Proposition-uniform estimate} yields
\begin{equation}\label{Uniform estimate'}
-\frac{1}{2n}<v_t<0\quad\mathrm{in}\ \Omega_t
\end{equation}
for all $t\geq 0$.

Let $B_\infty$ be the limiting geodesic ball and write $v_{B_\infty}$ for its transformed torsion function. Since its radius is at most $r_0(n)$, Proposition \ref{Proposition-ball case} gives $\Lambda_\infty>0$ in $B_\infty$. The smooth convergence of $\Omega_t$ to $B_\infty$, together with standard elliptic estimates, implies that $v_t$ converges locally smoothly to $v_{B_\infty}$. The strict positivity above therefore implies that $\Lambda_t>0$ on compact subsets of $B_\infty$ for all sufficiently large $t$. Meanwhile, since $\partial\Omega_t$ converges smoothly to $\partial B_\infty$, the boundary estimate in Proposition \ref{Proposition-boundary estimate} holds uniformly for all sufficiently large $t$. It follows that there exists $T>0$ such that $\Lambda_t>0$ for every $t\geq T$.

We next apply a continuity argument to show that $\Lambda_t$ is positive definite throughout the deformation. If $\Lambda_t$ failed to be positive definite at some positive time, let $t^\ast>0$ be the first such time as $t$ decreases from $T$. Then $\Lambda_{t^\ast}\geq 0$ by continuity and $\Lambda_{t^\ast}$ is degenerate at an interior point by the boundary convexity estimate (Proposition \ref{Proposition-boundary estimate}). In view of \eqref{Uniform estimate'}, the constant rank theorem (Proposition \ref{Proposition-constant rank theorem}) shows that $\Lambda_{t^\ast}$ has constant rank in $\Omega_{t^\ast}$, contradicting the fact that it is degenerate at an interior point but positive definite near the boundary. Hence $$\Lambda_t>0\quad\mathrm{in}\ \Omega_t$$ for every $t>0$. Letting $t\rightarrow 0^+$ yields $\Lambda_0\geq 0$ in $\Omega$, and using Proposition \ref{Proposition-boundary estimate} and Proposition \ref{Proposition-constant rank theorem} once again, we conclude that $$\nabla^2v-|\nabla v|g>0\quad\mathrm{in}\ \Omega.$$ In particular, $\nabla^2v>0$ in $\Omega$, so $v$ is strictly convex, which completes the proof.
\end{proof}

\section{Optimality of the exponent}\label{Section 7}

We conclude by showing that the exponent $1/2$ in Theorem \ref{Theorem-main theorem} is optimal, even among strictly horo-convex domains of arbitrarily small diameter. To be precise, given any $\alpha>1/2$ and $\varepsilon>0$, we construct a bounded smooth strictly horo-convex domain $\Omega\subset\mathbb{H}^n$ with $\mathrm{diam}(\Omega)<\varepsilon$ whose torsion function $u$ has the property that $u^\alpha$ is not concave.

By the sharpness result of Kennington \cite{Kennington1985} and a smooth approximation, there exists a bounded smooth domain $D\subset\mathbb{R}^n$ whose boundary has positive principal curvatures and whose torsion function $U$ satisfies $$\begin{cases}
\Delta U=-2 & \mathrm{in}\ D,\\
\hspace{0.3cm} U=0 & \mathrm{on}\ \partial D,
\end{cases}$$ and
\begin{equation}\label{Euclidean positive direction}
\nabla^2(U^\alpha)(e,e)>0\quad\mathrm{at}\ x_0
\end{equation}
for some $x_0\in D$ and $e\in\mathbb{R}^n$ with $|e|=1$.

We now transfer this example to hyperbolic space. Fix $o\in\mathbb{H}^n$, identify $T_o\mathbb{H}^n$ with $\mathbb{R}^n$ by a linear isometry, and use the exponential map $\exp_o: T_o\mathbb{H}^n\rightarrow\mathbb{H}^n$ to define $$F_\delta(x)=\exp_o(\delta x),\quad \Omega_\delta=F_\delta(D),\quad\mathrm{for}\ \delta>0.$$ On $D$, we introduce the rescaled pullback metric $g_\delta=\delta^{-2}F_\delta^\ast g$. If $g_{ij}(y)$ denote the coefficients of the hyperbolic metric in normal coordinates centered at $o$, then $$(F_\delta^\ast g)_{ij}(x)=\delta^2g_{ij}(\delta x)$$ and $$(g_\delta)_{ij}(x)=g_{ij}(\delta x).$$ Since the coefficients $g_{ij}$ are smooth and $g_{ij}(0)=\delta_{ij}$, the chain rule yields $$g_\delta\rightarrow g_{\mathbb{R}^n}\quad\mathrm{in}\ C^\infty(\overline{D})$$ as $\delta\rightarrow 0$.

Let $u_\delta$ be the torsion function of $\Omega_\delta$, and define its rescaled pullback to $D$ by $$U_\delta=\delta^{-2}u_\delta\circ F_\delta.$$ Since $F_\delta$ is an isometry from $(D,F_\delta^\ast g)$ onto $(\Omega_\delta,g)$ and $$\Delta_{F_\delta^\ast g}=\Delta_{\delta^2g_\delta}=\delta^{-2}\Delta_{g_\delta},$$ we obtain $$\Delta_{g_\delta}U_\delta=\Delta_{F_\delta^\ast g}(u_\delta\circ F_\delta)=(\Delta_g u_\delta)\circ F_\delta=-2\quad\mathrm{in}\ D,$$ with $U_\delta=0$ on $\partial D$. Consequently, the smooth convergence of the metrics and standard elliptic estimates yield $$U_\delta\rightarrow U\quad\mathrm{in}\ C_{\mathrm{loc}}^2(D),$$ and hence, by \eqref{Euclidean positive direction}, $$\nabla^2_{g_\delta}(U_\delta^\alpha)(e,e)>0\quad \mathrm{at}\ x_0$$ for all sufficiently small $\delta$. Moreover, since $u_\delta\circ F_\delta=\delta^2U_\delta$ and a constant rescaling of a metric does not change its Levi-Civita connection, we have $$F_\delta^\ast(\nabla_g^2u_\delta^\alpha)=\nabla_{F_\delta^\ast g}^2(u_\delta^\alpha\circ F_\delta)=\delta^{2\alpha}\nabla_{g_\delta}^2(U_\delta^\alpha).$$ Thus $\nabla_g^2u_\delta^\alpha$ has a positive direction at $F_\delta(x_0)$, and hence $u_\delta^\alpha$ is not concave in $\Omega_\delta$.

It remains to verify that $\Omega_\delta$ is strictly horo-convex and has arbitrarily small diameter. Since $\partial D$ has positive Euclidean principal curvatures and $g_\delta$ converges smoothly to the Euclidean metric on $\overline{D}$, the principal curvatures of $\partial D$ with respect to $g_\delta$ are bounded below by a constant $c>0$ for all sufficiently small $\delta$. Under the rescaling $F_\delta^\ast g=\delta^2g_\delta$, the principal curvatures are multiplied by $\delta^{-1}$, and the isometry property of $F_\delta$ therefore shows that the principal curvatures of $\partial\Omega_\delta$ are bounded below by $c/\delta>1$ when $\delta$ is sufficiently small. Hence $\Omega_\delta$ is strictly horo-convex. Finally, the identity $\mathrm{d}_g(o,F_\delta(x))=\delta|x|$, where $|\cdot|$ denotes the Euclidean norm, and the triangle inequality imply $$\mathrm{diam}(\Omega_\delta)\leq 2\delta\sup_{x\in D}|x|\rightarrow 0.$$ Taking $\delta$ sufficiently small, we obtain $\mathrm{diam}(\Omega_\delta)<\varepsilon$, while $u_\delta^\alpha$ is not concave. This completes the proof of the optimality of the exponent.

\bibliographystyle{plain}

\bibliography{mybibliography}

\end{document}